\documentclass[11pt]{amsart}

\usepackage{amsmath,amssymb,amscd,amsthm,amsfonts,amstext,amsbsy,mathrsfs,hyperref,upgreek,mathtools,stmaryrd,enumitem,bbm}

\usepackage{fontenc}
\usepackage{amsmath}
\usepackage{amsfonts}
\usepackage{amssymb}
\usepackage{amsthm}
\usepackage{graphicx}
\usepackage[english]{babel}
\usepackage{array}

\usepackage{ dsfont }

\usepackage[textsize=footnotesize,color=yellow!70, bordercolor=white]{todonotes}
\usepackage{hyperref}

\usepackage{amsmath}
\usepackage{amssymb}

\usepackage{amsthm}

\usepackage{ marvosym }

\usepackage{amsthm}
\usepackage{amsmath}
\usepackage{amssymb}
\usepackage{amsfonts}
\usepackage{amsxtra}
\usepackage{mathrsfs}
\usepackage{color}
\usepackage{graphicx}
\usepackage[all]{xy}

\usepackage[utf8x]{inputenc}

\usepackage{amsmath, amssymb, amsthm}
\usepackage{stmaryrd}
\usepackage{graphicx}
\usepackage{cancel}
\usepackage{txfonts}
\usepackage{mathrsfs}
\usepackage{textcomp}
\usepackage{relsize}
\usepackage{cite}

\newtheorem{definition2}{Definition}[section]

\theoremstyle{definition}
\newtheorem{defi}[definition2]{Definition}

\newtheorem{rema}[definition2]{Remark}

\theoremstyle{plain}
\newtheorem{theo}[definition2]{Theorem}

\newtheorem{coro}[definition2]{Corollary}

\newtheorem*{clai*}{Claim}

\DeclareMathOperator{\PP}{\mathbb P}

\newcommand{\cf}{\mathrm{cf}}

\newcommand{\with}{\mid}

\DeclareMathSymbol{\shortminus}{\mathbin}{AMSa}{"39}

\newcommand{\om}{\omega}

\newcommand{\powerset}{\mathcal{P}}

\newcommand{\subs}{\subseteq}

\newcommand{\h}{\mathfrak{h}}
\newcommand{\tfrak}{\mathfrak{t}}

\newcommand{\sfrak}{\mathfrak{s}}

\newcommand{\cc}{\mathfrak{c}}

\newcommand{\fin}{\text{fin}}

\newcommand{\fresh}[1]{\textsf{FRESH}}

\newcommand{\com}[1]{\textsf{COM}}

\newcommand{\hh}[1]{\mathfrak{h}(#1)}

\newcommand{\Pom}{\powerset(\omega)/\fin}

\newcommand{\coom}[1]{\textsf{COM}^{cof}}
\newcommand{\coombase}[1]{\textsf{COM}^{base+cof}}
\newcommand{\combase}[1]{\textsf{COM}^{base}}

\newcommand{\strcom}[1]{||_\lambda^{strong}}

\begin{document}

\title{Games on base matrices}

\author{Vera Fischer}
\address{Institute of Mathematics, University of Vienna, Kolingasse 14--16, 1090 Wien, Austria}
\email{vera.fischer@univie.ac.at}

\author{Marlene Koelbing}
\address{Institute of Mathematics, University of Vienna, Kolingasse 14--16, 1090 Wien, Austria}
\email{marlenekoelbing@web.de}

\author{Wolfgang Wohofsky}
\address{Institute of Mathematics, University of Vienna, Kolingasse 14--16, 1090 Wien, Austria}
\email{wolfgang.wohofsky@gmx.at}

\thanks{\emph{Acknowledgments.} The authors would like to thank the Austrian Science Fund (FWF) for the generous support through grants Y1012, I4039 (Fischer, Wohofsky) and P28420 (Koelbing). The second author is also grateful for the support by the  \"OAW Doc fellowship.}



\begin{abstract} Using a game characterization of distributivity, we show that base matrices 
for $\mathcal{P}(\omega)/\text{fin}$ 
of regular height larger than $\mathfrak{h}$ necessarily have maximal branches which are not cofinal. 
\end{abstract}

\maketitle

\section{Introduction}

A forcing $\PP$ is
\emph{$\delta$-distributive} if 
any system of $\delta$ many
maximal antichains
has a common refinement.
The \emph{distributivity} of a forcing notion~$\PP$,
denoted by~$\hh{\PP}$,
is the
least~$\lambda$
such that 
$\PP$ is not $\lambda$-distributive.
In particular, 
$\h(\Pom)$
is 
the classical cardinal characteristic $\h$. 
Note that
$\h(\PP)$
is
actually
the
least~$\lambda$
such that there is a system of $\lambda$ many
\emph{refining}
maximal antichains
without common refinement,
which
gives rise to 
the following 
definition:

\begin{defi}\label{defi:com}
We say that
$\mathcal{A}=\{A_\xi \with \xi<\lambda\}$
is a \emph{distributivity matrix for $\PP$ of height~$\lambda$}
if
\begin{enumerate}

\item[(1)] $A_\xi$ is a maximal antichain in $\PP$, for each $\xi < \lambda$,

\item[(2)] $A_\eta$ refines $A_\xi$ whenever $\eta \geq \xi$, 
i.e., for each $b \in A_\eta$ there exists $a \in A_\xi$ such that $b \leq a$, 
and

\item[(3)] there is no common refinement, i.e., there is no maximal antichain~$B$ which refines every $A_\xi$.

\end{enumerate}
\end{defi}

A special sort of distributivity matrices have
been
considered
in the seminal paper~\cite{basematrix} of Balcar, Pelant, and Simon, where $\h$ has been introduced:

\begin{defi}\label{defi:base_matrix}
A distributivity matrix $\{A_\xi \with \xi<\lambda\}$ for~$\PP$ is a \emph{base matrix} if
$\bigcup_{\xi < \lambda} A_\xi$ is dense in~$\PP$,
i.e., for each
$p \in \PP$ there is $\xi < \lambda$ and $a \in A_\xi$ such that $a \leq p$.
\end{defi}

In~\cite{basematrix}, the 
famous 
base matrix theorem 
has been shown: 
there exists a base matrix
for $\Pom$ of height $\h$. 
A more general version for a wider class of forcings has been given
in~\cite[Theorem~2.1]{basematrix_general}.

Due to its refining structure, a distributivity matrix $\{A_\xi \with \xi<\lambda\}$ can be viewed as a tree, with level $\xi$ being $A_\xi$.
Let
us say that~$\langle a_\xi \with \xi < \delta \rangle$ is a
\emph{branch 
of 
the
distributivity matrix~$\{A_\xi \with \xi<\lambda\}$} if $a_\xi \in A_\xi$ for each $\xi < \delta$, and
$a_\eta \leq a_\xi$ for each $\xi \leq \eta < \delta$.
We say that the branch is \emph{maximal} if there is no branch 
of~$\{A_\xi \with \xi<\lambda\}$
strictly extending it.
If $\delta= \lambda$, the branch $\langle a_\xi \with \xi < \delta \rangle$ is called \emph{cofinal} in $ \{A_\xi \with \xi<\lambda\}$.

The structure of base matrices for $\Pom$ has been investigated in the literature. 
Note that each maximal branch 
of 
a distributivity matrix for $\Pom$ which is not cofinal is a tower. So if there are no towers of length strictly less than~$\h$, i.e., if $\tfrak = \h$, all maximal branches of a distributivity matrix of height~$\h$ are cofinal.
On the other hand, Dow showed 
that in the Mathias model,
there exists a base matrix of
height~$\h$
without cofinal branches (see~\cite[Lemma~2.17]{Dow_tree_pi_bases}).
It is actually consistent that
\emph{no} 
base matrix of height~$\h$
has cofinal branches.
This
was
proved by Dordal
by constructing
a model in which $\h$ does not belong to the tower spectrum (see \cite{Dordal_model} 
or\footnotemark~\cite[Corollary~2.6]{Dordal_tower_spectrum}).

\footnotetext{Dordal's original model 
(in which $\cc = \om_2$) is presented in~\cite{Dordal_model}, 
whereas 
\cite[Corollary~2.6]{Dordal_tower_spectrum} 
is 
a more general
result 
which also 
gives models satisfying $\h = \cc > \om_2$ 
(but is, interestingly enough,
easier to prove).}

In \cite{our_COM}, the authors of this paper have shown that consistently there exists a distributivity matrix 
for $\Pom$ 
of regular height larger than $\h$ in which all maximal branches are cofinal. 

In \cite{Brendle_Base}, Brendle has shown that 
if 
$\lambda\leq\cc$ 
is regular and greater or equal than the splitting number~$\sfrak$ (or,  alternatively, there exists no strictly $\subseteq^*$-decreasing sequence of length $\lambda$), then there exists a base matrix 
for $\Pom$ 
of height~$\lambda$. In particular, there always exists a base matrix of height~$\cc$ provided that $\cc$ is regular. 
He mentions that in the Cohen and random models base matrices of height larger than $\h$ necessarily have maximal branches which are not cofinal (in fact, there are no strictly $\subseteq^*$-decreasing sequences of length larger than $\omega_1$).

We will show below 
that,
in ZFC, 
any base matrix 
for $\Pom$ 
of regular 
height larger than $\h$ has maximal branches which are not cofinal.

\section{Main Result}

In the proof of Theorem~\ref{theo:theta_Beweis}, 
we will use a game characterization of being $\lambda$-distributive. 
It generalizes the game characterization for being $\om$-distributive which can be found in~\cite[Lemma~30.23]{Jech}.

\begin{defi}
Let $\PP$ be a forcing notion.
Let $G_\lambda(\PP)$ denote the
\emph{$\lambda$-distributivity game}
(which has
length~$\lambda$):
\begin{center}
\begin{tabular}[h]{l|llllllllllll}
I 	& $a_0$ 	& 			& $a_1$ 	&    		&  $\dots$ 		&			&			&$a_{\mu+1}$& 				&$\dots$& &  \\
\hline
II 	&  			& $b_0$ 	& 			& $b_1$ 	&   				& $\dots$ 	&$b_\mu$	&				& $b_{\mu+1}$		& 		& $\dots$&	\\
\end{tabular}
\end{center}
The players alternately pick conditions 
in $\PP$
such that the resulting sequence is decreasing, i.e.,
$b_j\leq a_i$ and $a_{i+1}\leq b_i$ for every $i\leq j <\lambda$.  Player~I starts the game, and at limits $\mu$, Player~II has to play. If Player~II cannot play at limits (because the sequence played till then has no lower bound),
the game ends and
Player~I wins immediately.
If the game continuous for $\lambda$ many steps, Player~II wins if and only if there exists
a 
$b\in \PP$ with $b\leq a_i$ for every successor $i<\lambda$.
\end{defi}

Recall that, by definition, a forcing $\PP$ is
\emph{${\leq}\lambda$-strategically closed} if Player~II has a winning strategy in $G_\lambda(\PP)$.
A slightly weaker property turns out to be equivalent to being 
$\lambda$-distributive;  
this was shown by Foreman in 
\cite[Theorem on page~718]{Foreman}:

\begin{theo}\label{theo:game}
 The following are equivalent:
\begin{enumerate}
\item Player~I has no winning strategy in the game $G_\lambda(\PP)$.
\item $\PP$ is $\lambda$-distributive.
\end{enumerate}
\end{theo}

We now show that a base matrix 
for~$\Pom$ 
of regular height larger 
than~$\h$ necessarily
has 
(through every node)
branches which are dying out 
(see Corollary~\ref{coro:theta_Beweis_Pom}).
In fact, we show the following more general theorem:

\begin{theo}\label{theo:theta_Beweis}
Let $\mathcal{A}=\{A_\xi \with \xi<\lambda\}$ 
be a base matrix 
for~$\PP$ 
of
regular
height $\lambda > \h(\PP)$. 
Then
there is a maximal branch 
of~$\mathcal{A}$ which is not cofinal.
\end{theo}

\begin{proof}
Assume towards contradiction that
every maximal branch 
is cofinal.
By definition, 
$\PP$ is \emph{not} $\h(\PP)$-distributive.
Therefore, using the game characterization of distributivity (see Theorem~\ref{theo:game}),
Player~I has a winning strategy $\sigma$ in
the game~$G_{\h(\PP)}(\PP)$.

Consider the following run of the game of (full) 
length~$\h(\PP)$:
\begin{center}
\begin{tabular}[h]{l|lllllllllll}
I 	& $b_0$ 	& 			& $b_1$ 	&    		&  $\dots$ 		&			&			&$b_{\mu+1}$& 				&$\dots$ & \\
\hline
II 	&  			& $a_0$ 	& 			& $a_1$ 	&   				& $\dots$ 	&$a_\mu$	&				& $a_{\mu+1}$ 		& 	& $\dots$		\\
\end{tabular}
\end{center}
where Player~I plays according to $\sigma$
(i.e.,
$b_0 = \sigma(\langle \rangle)$
and
$b_{i+1} =
\sigma(\langle b_0, a_0, \dots, a_i \rangle)$
for each $i < \h(\PP)$),
and Player~II plays as follows (where the $a_i$ are going to be in the matrix for each $i < \h(\PP)$). For successors $i$ (and for $i = 0$), let 
$a_{i} \leq b_i$ with
$a_{i}$ in the matrix; this is possible,
because
it
is a base matrix.
For
limit~$\mu\leq \h(\PP)$,
the following holds by induction:
$\langle a_i \with i < \mu \rangle$
is
a $\leq$-decreasing sequence 
such that
$a_i$ is in the matrix for each~$i < \mu$.
So (for $\mu < \h(\PP)$)
Player~II can play
a lower bound~$a_\mu$ in the matrix,
by the following claim.

\begin{clai*}
The sequence $\langle a_i \with i < \mu \rangle$
has a lower bound  
in the matrix.
\end{clai*}

\begin{proof}
We can assume that the sequence is not eventually 
constant. 
Moreover, we can assume that it is 
strictly decreasing.
It is easy to check that there is a strictly increasing
sequence
$\langle \xi_i \with i < \mu \rangle \subs \lambda$
with
$a_i \in A_{\xi_i}$ for each $i < \mu$.
Then
$\sup(\{\xi_i \with i < \mu \}) < \lambda$, because
$\mu \leq \h(\PP) < \lambda$ and $\lambda$ is regular.
So the corresponding branch is not cofinal in the matrix, hence it is not maximal by assumption.
Consequently, there
exists an~$a$ in the matrix such that 
$a \leq a_i$ for each $i < \mu$.
\end{proof}

Finally, for $\mu = \h(\PP)$, the claim yields a 
lower bound of
$\langle a_i \with i < \h(\PP) \rangle$,
witnessing that
Player~II wins this run of the game.
This contradicts that $\sigma$ is a winning strategy for Player~I.
\end{proof}

\begin{rema}
In the above theorem, 
the assumption that $\lambda>\h(\PP)$ can be replaced by the
weaker assumption
that
$\PP$ is,
for some
$\nu<\lambda$,
not ${\leq}\nu$-strategically closed.
In the proof,
one can
turn
(using a well-order on the base matrix)
the description of the moves of
Player~II
into a strategy for Player~II, which is then a winning strategy.

Also, note that the proof still works 
for singular~$\lambda$ as long as 
$\cf(\lambda) > \h(\PP)$  
(or if $\PP$ is not ${\leq}\nu$-strategically closed 
for some
$\nu < \cf(\lambda)$).
\end{rema}

For the important case of~$\Pom$, we can now derive the following:

\begin{coro}\label{coro:theta_Beweis_Pom}
Let $\mathcal{A}=\{A_\xi \with \xi<\lambda\}$ be a base matrix 
for $\Pom$ 
of
regular
height~$\lambda > \h$.
Then
for every $a\in \bigcup_{\xi <\lambda} A_\xi$, 
there is a maximal branch of~$\mathcal{A}$ containing~$a$ which is not cofinal.
\end{coro}

\begin{proof}
Fix~$a$ in the matrix
(i.e., $a \in \bigcup_{\xi <\lambda} A_\xi$). 
Let $\PP := \{ b \with b \subs^* a \}$ be the part of $\Pom$ below~$a$. 
Recall that $\Pom$ is homogenous, hence $\h(\PP)=\h(\Pom)=\h$. 
Note that the part of $\mathcal{A}$ below~$a$ is a base matrix for $\PP$ of height~$\lambda$. Therefore, by Theorem~\ref{theo:theta_Beweis}, 
there is 
a maximal branch 
of~$\mathcal{A}$ containing~$a$
which is not cofinal. 
\end{proof}

So the above theorem actually says that distributivity matrices 
for $\Pom$ 
of regular height 
larger than~$\h$ cannot 
simultaneously 
have only cofinal maximal branches and be a base matrix. 
Therefore, Brendle's theorem from~\cite{Brendle_Base} together with the above theorem shows that there are distributivity matrices 
for $\Pom$ 
of regular height larger than~$\h$ with maximal branches which are not cofinal (at least if $\cc > \h$ is regular or $\sfrak < \cc$). 
On the other hand, the above theorem shows that the generic distributivity matrix of regular height larger than~$\h$ from~\cite{our_COM} cannot be a base matrix because all its maximal branches are cofinal (this can also be seen by analyzing the forcing construction, see the end of \cite[Section 7.1]{our_COM}).

Further note that 
in the model from~\cite{our_COM}, 
there are both kinds of distributivity matrices 
for $\Pom$ 
of regular height larger than~$\h$: 
matrices all whose maximal branches are cofinal, and 
matrices with maximal branches which are not cofinal.

\bibliography{bib_files_Vera_h}
\bibliographystyle{alpha}

\end{document}